\documentclass[12pt]{article}               
\usepackage{geometry}
\usepackage{cite}
\usepackage{multirow}
\usepackage{amsmath}
\usepackage{amssymb}
\usepackage[usenames]{color}
\usepackage{graphicx}          
 \usepackage{amsthm}                              
\usepackage{epsfig}    
 \usepackage{helvet}

\newtheorem{theorem}{Theorem}

\title{\large Stabilization of Crystallization Models \\ Governed by Hyperbolic Systems}

\author{Alexander Zuyev\thanks{
Max Planck Institute for Dynamics of Complex Technical Systems, Magdeburg, Germany
({\tt\small zuyev@mpi-magdeburg.mpg.de, benner@mpi-magdeburg.mpg.de}) \newline
$^{**}$Institute of Applied Mathematics and Mechanics, National Academy of Sciences of Ukraine, Slovyansk
        \newline
} $^{, **}$ and Peter Benner$^{*}$}
\date{}

\begin{document}

\maketitle
\thispagestyle{empty}

\begin{abstract}
This paper deals with mathematical models of continuous crystallization described by hyperbolic systems of partial differential equations coupled with ordinary and integro-differential equations. The considered systems admit nonzero steady-state solutions with constant inputs.
To stabilize these solutions, we present an approach for constructing control Lyapunov functionals based on quadratic forms in weighted $L^2$-spaces.
It is shown that the proposed control design scheme guarantees exponential stability of the closed-loop system.
\end{abstract}

\section{Introduction}

The study of the literature in the field of mathematical control theory for distributed parameter systems shows that the development of control design techniques
is to a considerable extent influenced by problems of chemical engineering.
Important examples in this area come from mathematical models of distillation, chromatography, and crystallization processes governed by hyperbolic systems of partial differential equations~\cite{M1990,RL1988,R2000,GK2004}. For a moving bed chromatography with considerable apparent dispersion coefficients, a parabolic-type equilibrium dispersive model is also available for theoretical studies (cf.~\cite{GFS2006}).

The main challenge concerning applications of Lyapunov's direct method to quasilinear hyperbolic systems is related to the construction of a Lyapunov functional with negative definite time derivative.
For a class of hyperbolic systems with boundary control, strict control Lyapunov functionals have been proposed in~\cite{BC2016}.
The construction of these functionals requires that the solution of an associated ordinary differential equation should be defined on a prescribed interval.
The proposed method has been applied, in particular, to stabilize the equilibrium of the Saint-Venant equations represented as a 2x2 hyperbolic system. An important feature of this approach relies on the possibility of studying control systems with non-uniform steady states.

It should be mentioned that the backstepping approach~\cite{K1992} has been already applied for solving the stabilization problem for several classes of distributed parameter systems~\cite{KS2008}.
In particular, this approach has been developed in~\cite{LB2015} for problems of trajectory generation and tracking for linear 2x2 hyperbolic systems of partial differential equations with boundary inputs and outputs.
In the paper~\cite{DK2019}, the backstepping approach is applied to the output regulation problem for a class of coupled linear parabolic  integro-differential equations.
To the best of our knowledge, this design methodology has not been applied for the exponential stabilization of integro-differential models of cooling and preferential crystallization so far.

The flatness based approach~\cite{F1995} is shown to be a powerful method for
nonlinear models of chemical engineering with known flat output (see, e.g.,~\cite{RRZ1996}).
This approach is also applicable for the trajectory tracking problem of distributed parameter systems with integral terms,
including a class of parabolic-like linear Volterra partial integro-differential equation with boundary control~\cite{M2016}.
However, the question of checking flatness for general classes of systems and constructing a flat output remains open up to now.

Although there are well-established control design techniques for hyperbolic systems with boundary controls~\cite{BC2016},
mathematical models of crystallization processes require the analysis of coupled systems of first-order quasilinear partial and ordinary differential equations with integral terms.
An solution of the local steering problem for a finite-dimensional nonlinear crystallization model
has been proposed in~\cite{ZB2016} by exploiting the Lie bracket approximation techniques with open-loop controls (cf.~\cite{ZG2017}).
For infinite-dimensional crystallization models, the design of stabilizing feedback laws remains an open problem.
Our paper aims at solving this problem for the classes of continuous crystallization models introduced in~\cite{VR2001} and~\cite{Q2008}.

\section{Continuous Crystallization Model}\label{cryst1d}
Consider a continuous cooling crystallization model described by the population balance and mass balance equations as follows~\cite{VR2001}:
\begin{equation}\label{PB}
\begin{aligned}
& \frac{\partial n(x,t)}{\partial t} + G(x,c)\frac{\partial n(x,t)}{\partial x} = {v} \psi(x)n(x,t),\quad x\in [0,\ell], \\
&
n(0,t) = B(c)/G(0,c),\;\quad G>0,\;B\ge 0,
\end{aligned}
\end{equation}
\begin{equation}\label{MB}
\begin{aligned}
\frac{dc}{dt} &= (\rho_0-c)\left(v+\frac{d \ln\varepsilon(n(\cdot,t))}{dt}\right)\\
& +\frac{v}{\varepsilon(n(\cdot,t))}\left(u_f -\rho_0-\rho_0 k_v\int_0^\ell \phi(x)n(x,t)dx\right), \\
& \varepsilon(n(\cdot,t))=1-k_v\int_0^\ell x^3 n(x,t)dx>0.
\end{aligned}
\end{equation}
Here, the crystal size distribution function $n(x,t)\in {\mathbb R}^+=[0,+\infty)$ denotes the expected number of crystals of size $x\in [0,\ell]$ at time $t\ge 0$.
Equation~\eqref{MB} relates the solid phase with
the mass concentration of solute $c=c(t)\ge 0$ in the liquid phase, where $\rho_0>0$ is the crystal density, $v>0$ is the flow-rate parameter, and $\varepsilon (n(\cdot,t))$ is the void fraction.
The crystallization process is controlled by the mass concentration of the solute in the feed $u_f\ge 0$.
In this paper, we allow the growth rate $G(x,c)$ to depend on the crystal size.
We refer the reader to~\cite{VR2001} for information about the nucleation rate $B(c)$,  classification functions $\psi(x)$, $\phi(x)$, and the volumetric shape factor $k_v$.
The functions $G(x,c)$ and $B(c)$ are assumed to be continuously differentiable in their domains of definition, while $\psi(x)$ and $\phi(x)$ are piecewise continuous.

Equations~\eqref{PB} and~\eqref{MB} admit the steady-state solution $n(x,t)=\bar n(x)$ and $c(t)=\bar c$ with a constant control $u_f = \bar u_f$, where
\begin{equation}\label{S}
\begin{aligned}
\bar n(x)&= \frac{B(\bar c)}{G(0,\bar c)} \exp \left\{v \int_0^x \frac{\psi(y) dy}{G(y,\bar c)}\right\},\quad x\in [0,\ell],\\
\bar c & = \rho_0 + \frac{1}{\varepsilon(\bar n(\cdot))}\left(\bar u_f - \rho_0 - \rho_0 k_v\int_0^\ell \phi(x)\bar n(x) dx \right).
\end{aligned}
\end{equation}
Our goal is to stabilize the above equilibrium by a state feedback law. By performing the change of variables
$$
\begin{aligned}
& n(x,t) = \bar n(x)+w(x,t),\\
& c(t)=\bar c+s(t),\\
& u_f=\bar u_f + u,
\end{aligned}
$$
we represent the linear approximation of~\eqref{PB} and~\eqref{MB} in a neighborhood of~\eqref{S} as follows:
\begin{equation}\label{lambda}
\begin{aligned}
&\frac{\partial w(x,t)}{\partial t} = -g(x) w'(x,t)+ {v} \psi(x)w(x,t)-g_c(x) \bar n'(x) s(t),\quad x\in [0,\ell],
 \\
& w(0,t) = \alpha s(t), \\
& \frac{ds(t)}{dt} = -k_0 s(t) + k_1 w(\ell,t)+ \int_0^\ell \theta (x) w(x,t)dx + b u ,
\end{aligned}
\end{equation}
where the prime stands for the derivative with respect to $x$,
\begin{equation}
\begin{aligned}
& g(x)=G(x,\bar c),\; g_c(x) = \left. \frac{\partial G(x,c)}{\partial c}\right|_{c=\bar c},\;\alpha = \frac{d}{dc}\left.\left(\frac{B(c)}{G(0,c)}\right)\right|_{c=\bar c}, \\
& k_0 = v+\frac{(\bar c - \rho_0)k_v}{\varepsilon(\bar n)}\int_0^\ell x^3 g_c \bar n' dx,\; k_1 = \frac{(\rho_0-\bar c)k_v \ell^3}{\varepsilon(\bar n)}g(\ell),\\
& \theta(x) = \frac{k_v}{\varepsilon (\bar n)}\left\{(\bar c-\rho_0) \bigl((x^3g(x))'+vx^3\psi(x)\bigr)-v\rho_0\phi(x)+ v\beta x^3 \right\},\\
& \beta = \bar u_f - \rho_0-\rho_0 k_v \int_0^\ell \phi(x) \bar n(x) dx,\; b=\frac{v}{\varepsilon(\bar n)}.
\end{aligned}
\label{cs_coeff}
\end{equation}
Note that the coefficients and parameters of~\eqref{lambda} satisfy the following inequalities for the realistic crystallization example considered in~\cite{VR2001}:
$$
\rho_0>\bar c>0,\;k_0>0,\;k_1>0,\; \alpha>0,\; b>0,\; g>0,\; g_c >0,\; \psi\le 0.
$$
Moreover, the growth rate $G$ is independent of $x$ and affine in $c$ for the example of~\cite{VR2001}.

\section{Control Design}
Consider a control Lyapunov functional candidate
\begin{equation}
V = \frac12\int_0^\ell \rho(x){w}^2(x,t)dx + \frac{\gamma}{2} s^2(t),
\label{V}
\end{equation}
where $\rho(x)>0$ is a continuous density function to be defined later, and $\gamma$ is a positive constant.
The time derivative of $V$ along the classical solutions of~\eqref{lambda} takes the form
\begin{equation}\label{Vdot}
\begin{aligned}
\dot V = & \frac{1}{2}\int_0^\ell \left\{ (\rho g)' + 2v \rho \psi\right\} w^2 dx -\left.\frac{\rho g w^2}{2}\right|_{x=\ell}  -\left(\gamma  k_0- \frac{\rho(0)g(0)\alpha^2}{2}
\right)s^2 + \gamma b s u
\\
& + s\left(\int_0^\ell (\gamma \theta - \rho g_c \bar n')w\, dx + \gamma k_1 w|_{x=\ell}\right).
\end{aligned}
\end{equation}
The above formula is obtained by performing the integration by parts with regard to the boundary condition $w(0,t) = \alpha s$.
Note that by constructing the Lyapunov functional~\eqref{V} we aim to achieve strong stability in the corresponding weighted $L^2$-space.
A weaker stability notion with respect to some integral measure has been analyzed in the paper~\cite{PK2014} for a population balance model,
which is relevant to the stability problem with respect to two measures (cf.~\cite{MS2008}) or partial stability concept~\cite{Z2015,Z2000}.

It will be shown in the sequel that $\dot V$ can be made negative definite in an appropriate state space with the following feedback law:
\begin{equation}
u = -\frac{1}{\gamma b} \left(\varkappa s + \int_0^\ell (\gamma \theta - \rho g_c \bar n')w\, dx + \gamma k_1 w|_{x=\ell}\right),
\label{feedback1}
\end{equation}
where $\varkappa\in \mathbb R$ is a design parameter.
To answer the question whether the proposed feedback control~\eqref{feedback1} stabilizes the trivial solution of~\eqref{lambda},
we take the density function $\rho(x)>0$ as a solution of the ordinary differential equation
\begin{equation}\label{ODE}
\frac{d}{dx}(\rho(x) g(x)) + 2v \psi(x)\rho(x) = -h(x) \rho(x) \quad x\in [0,\ell],
\end{equation}
with some continuous function $h(x)>0$ to be defined on $[0,\ell]$.
The above equation is a particular case of the differential inequality proposed in~\cite{ZKB2017}.

Straightforward computations show that the general solution of~\eqref{ODE} is
\begin{equation}
\rho(x) = \bar \rho\exp\left\{-\int_0^x \frac{2v\psi(y)+g'(y)+h(y)}{g(y)}dy\right\},\quad \bar \rho>0.
\label{rho}
\end{equation}
Then the substitution of formulas~\eqref{feedback1} and~\eqref{ODE} into~\eqref{Vdot}  yields the time derivative of $V$ along the trajectories of the closed-loop system:
\begin{equation}
\dot V = -\frac{1}{2}\int_0^\ell\rho h w^2 dx -\left.\frac{\rho g w^2}{2}\right|_{x=\ell}  -\left(\varkappa+\gamma  k_0- \frac{\rho(0)g(0)\alpha^2}{2}
\right)s^2.
\label{Vdot_cs}
\end{equation}

\section{Stability Analysis}
To analyze stability properties of the above control system, we first perform the change of variables
$$
w(x,t) = \tilde w(x,t) + \alpha s(t).
$$
This allows to rewrite~\eqref{lambda} as a system with zero boundary condition at $x=0$:
\begin{equation}\label{lambda0}
\begin{aligned}
\frac{\partial \tilde w}{\partial t}= &- g{\tilde w}'+{v} \psi \tilde w - \alpha\int_0^\ell \theta(y)\tilde w(y,t)dy-\alpha k_1 \tilde w(\ell,t)\\
 &+ \left(\alpha k_2 + \alpha v \psi - g_c \bar n'\right)s-\alpha b u,\;\; x\in [0,\ell], \\
 \tilde w|_{x=0}   &=  0, \\
 \frac{ds}{dt} = & - k_2 s + k_1 \tilde w|_{x=\ell}+ \int_0^\ell \theta (x) \tilde w(x,t)dx + b u,
\end{aligned}
\end{equation}
where
$$
k_2 =k_0 -\alpha k_1 - \alpha\int_0^\ell \theta(x)dx.
$$

Let the function $\rho\in C^1[0,\ell]$ be defined by~\eqref{rho}, and let $L_\rho^2(0,\ell)$ denote the weighted $L^2$-space such that the inner product of $\eta_1,\eta_2\in L_\rho^2(0,\ell)$ is given by
$$
\left<\eta_1,\eta_2\right>_{L^2_\rho(0,\ell)} = \int_0^\ell \eta_1(x)\eta_2(x)\rho(x)dx.
$$
We also introduce the linear space
$$
H= \left\{\left.\xi=\begin{pmatrix}\eta \\ s\end{pmatrix}\,\right|\, \eta\in L^2_\rho(0,\ell),\; s\in{\mathbb R} \right\}
$$
with the following inner product of elements $\xi_1= \begin{pmatrix}\eta_1 \\ s_1\end{pmatrix}\in H$ and $\xi_2= \begin{pmatrix}\eta_2 \\ s_2\end{pmatrix}\in H$:
$$
\left<\xi_1,\xi_2\right>_H= \left<\eta_1+\alpha s_1,\eta_2+\alpha s_2\right>_{L^2_\rho(0,\ell)} + \gamma s_1 s_2.
$$
It is easy to see that $H$ is a Hilbert space if $\gamma>0$ .

Then system~\eqref{lambda0} can be represented as the abstract differential equation
\begin{equation}\label{cs_L2}
\frac{d}{dt}\xi(t) = A \xi(t) + B u,\quad \xi(t)\in H,\; u\in \mathbb R,
\end{equation}
with the unbounded linear operator $A:D(A)\to H$ defined by
\begin{equation}
D(A)=\left\{\left.\xi=\begin{pmatrix}\eta \\ s\end{pmatrix}\in H \,\right|\, \eta\in H^1(0,\ell),\; \eta(0)=0 \right\},
\label{A_op}
\end{equation}
$$\small
\xi=\begin{pmatrix}\eta \\ s\end{pmatrix} \mapsto A\xi = \begin{pmatrix}-g \eta' + v\psi \eta - \alpha \int_0^\ell \theta(y)\eta(y)dy - \alpha k_1 \eta(\ell) + s(\alpha k_2 + \alpha v \psi - g_c \bar n')\\
- k_2 s + k_1 \eta(\ell)+ \int_0^\ell \theta (y) \eta(y)dy
 \end{pmatrix},
$$
and
\begin{equation}
B = \begin{pmatrix} -\alpha b \\ b\end{pmatrix}\in H.
\label{B_op}
\end{equation}
Here $H^1(0,\ell)$ denotes the Sobolev space.

The feedback law~\eqref{feedback1} can be written in the operator form as
\begin{equation}
u=K\xi,
\label{K_op}
\end{equation}
where the linear functional $K:D(K)\subset H \to {\mathbb R}$ acts as
$$
\xi=\begin{pmatrix}\eta \\ s\end{pmatrix} \mapsto K\xi = -\frac{1}{\gamma b} \left(\varkappa s + \int_0^\ell (\gamma \theta(x) - \rho(x) g_c(x) \bar n'(x))\eta(x) dx + \gamma k_1 \eta(\ell)\right).
$$
We formulate the main stability result for the closed-loop system~\eqref{cs_L2},~\eqref{K_op} as follows.

\begin{theorem}\label{thm1}
Let the linear operator $\tilde A:D(A)\to H$ be defined as $\tilde A = A+BK$, where $A$, $B$, and $K$ are given by~\eqref{A_op},~\eqref{B_op}, and~\eqref{K_op}, respectively. Assume, moreover, that the function $\rho\in C^1[0,\ell]$ is defined by~\eqref{rho} with some $h\in C[0,\ell]$ and
\begin{equation}
\bar \rho>0,\; \gamma>0,\; \varkappa > \frac{\bar\rho g(0)\alpha^2}{2}-\gamma  k_0,\; g(\ell)>0,\; h(x)>0\;\;{\text for all}\;x\in [0,\ell].
\label{stabcond}
\end{equation}
Then the abstract Cauchy problem
\begin{equation}
\begin{aligned}
& \frac{d}{dt} \xi(t) = \tilde A \xi(t),\quad t\ge 0,\\
& \xi(0)=\xi_0 \in H,
\end{aligned}
\label{Cauchy_op}
\end{equation}
is well posed (in the sense of mild solutions), and the trivial solution of~\eqref{Cauchy_op} is exponentially stable, i.e.
\begin{equation}
\|\xi(t)\|_H \le \|\xi_0\|_H e^{-\omega t} \quad \text{for all}\;\;\xi_0\in H,\; t\ge 0,
\label{exp_op}
\end{equation}
with some $\omega>0$.
\end{theorem}

\begin{proof}
A straightforward computation shows that
\begin{equation}
\begin{aligned}
\left<\xi,\tilde A \xi\right>_H=&
-\frac{1}{2}\int_0^\ell\rho(x) h(x) \eta^2(x) dx -\frac{\rho(\ell) g(\ell) \eta^2(\ell)}{2} \\
& -\left(\varkappa+\gamma  k_0- \frac{\bar\rho g(0)\alpha^2}{2}
\right)s^2,
\end{aligned}
\label{diss_ineq}
\end{equation}
for all $\xi$ from the dense set $D(\tilde A)=D(A)\subset H$.
If the conditions~\eqref{stabcond} hold then $\left<\xi,\tilde A \xi\right>_H\le 0$ for all $\xi\in D(\tilde A)$, which proves that the operator $\tilde A$ is dissipative in~$H$.
It can also be shown that $\tilde A$ is closed, and $A - \lambda I$ is surjective for $\lambda>0$.
Hence, $\tilde A$ generates the $C_0$-semigroup of contractions $\{e^{t\tilde A}\}_{t\ge 0}$ on $H$ by the Lumer--Phillips theorem (cf.~\cite{P1983,B2010}).
The Cauchy problem~\eqref{Cauchy_op} is thus well-posed on $t\ge 0$, and its mild solutions are defined by
$$
\xi(t)= e^{t\tilde A}\xi_0,\quad \xi_0 \in H,\; t\ge 0.
$$

To prove the exponential decay estimate~\eqref{exp_op}, we analyze the behavior of
$$
V(\xi(t)) = \frac{1}{2}\|\xi(t)\|^2_H
$$
along the solutions of~\eqref{Cauchy_op}.
The above $V(\xi(t))$ plays the same role for the abstract problem~\eqref{Cauchy_op} as the Lyapunov functional~\eqref{V} for the closed-loop system~\eqref{lambda},~\eqref{feedback1}.

If $\xi(t)$ is a classical solution of~\eqref{Cauchy_op} (i.e. $\xi(t)\in D(\tilde A)$ for all $t\ge 0$), then $\frac{d}{dt}V(\xi(t))=\left<\xi(t),\tilde A \xi(t)\right>_H\le 0$. Moreover, the quadratic functional~\eqref{diss_ineq} is negative definite with respect to the norm $\|\cdot \|_H$ if the conditions~\eqref{stabcond} are satisfied, which means that
\begin{equation}\label{comparison_ineq}
\frac{d}{dt}V(\xi(t))=\left<\xi(t),\tilde A \xi(t)\right>_H\le -\frac{\delta}2 \|\xi(t) \|_H^2 = - \delta V(\xi(t)) \;\;\text{for}\;\; \xi(t)\in D(\tilde A)
\end{equation}
with some constant $\delta>0$.
Then~\eqref{exp_op} follows from~\eqref{comparison_ineq} and the Gr{\"o}nwall--Bellman inequality with $\omega=\delta/2>0$.
\end{proof}

\section{Preferential Crystallization Model}
Consider the $2x2$ hyperbolic system with one spatial variable that describes the preferential crystallization of enantiomers~\cite{Q2008,ZB2018}:
\begin{equation}\label{control2x2}
\begin{aligned}
&\frac{\partial n_k(x,t)}{\partial t} + G_k(S_k) \frac{\partial n_k(x,t)}{\partial x} = \psi(x) n_k(x,t),\quad x\in [0,\ell],\;t\ge 0, \\
&
\left. G_k (S_k)n_k \right|_{x=0} = B_k(S_k),\quad k=1,2,
\end{aligned}
\end{equation}
where $n_1(x,t)\ge 0$ and $n_2(x,t)\ge 0$ are the crystal size distributions for the preferred and counter enantiomers, respectively.
Here $G_k: [1,+\infty)\to {\mathbb R}^+$ characterizes the growth rate of crystals and $B_k:[1,+\infty)\to {\mathbb R}^+$ describes the nucleation rate of particles of minimum size for the $k$-th enantiomer.
These functions depend on the relative supersaturations $S_1\ge 1$ and $S_2\ge 1$ of the preferred and counter enantiomers,
which are mutually controlled by using the balance between the incoming and outgoing mass fluxes in the liquid phase.
It is assumed that $B_k$ and $G_k$ are differentiable and strictly increasing functions in their domain of definition such that $B_k(1)=0$ and $G_k(0)=0$ for $k=1,2$.
The classification function $\psi(x)$ describing the dissolution of particles below some critical values is assumed to be piecewise continuous on~$[0,\ell]$.

It is easy to see that system~\eqref{control2x2} with $ S_k = \bar S_k={\rm const}>1$ has the equilibrium $n_k(x,t)=\bar n_k(x)$,
\begin{equation}
\bar n_k(x)= \frac{\bar B_k}{\bar G_k}\exp\left\{\frac{1}{\bar G_k}\int_0^x \psi(y)dy \right\}, \quad k=1,2,
\label{equilibrium2x2}
\end{equation}
where
$$
\bar B_k = B_k(\bar S_k)>0,\;\bar G_k = G_k(\bar S_k)>0.
$$

To study the crystallization dynamics in a neighborhood of the steady state~\eqref{equilibrium2x2},
we rewrite system~\eqref{control2x2} with respect to $w_k(x,t)=n_k(x,t)-\bar n_k(x)$
as follows:
\begin{equation}
\begin{aligned}
& \frac{\partial  w_k(x,t)}{\partial t} = - (\bar G_k+\Delta G_k) \frac{\partial  w_k(x,t)}{\partial x} + \psi(x)  w_k(x,t)-\frac{\Delta G_k}{\bar G_k}\psi(x)\bar n_k(x), \\
& \left. (\bar G_k + \Delta G_k) { w_k} \right|_{x=0} = \Delta B_k - \Delta G_k \bar B_k/{\bar G_k},\quad k=1,2,
\end{aligned}
\label{sys2x2var}
\end{equation}
where $\Delta B_k = B_k-\bar B_k$ and $\Delta G_k = G_k-\bar G_k$.
Note that the deviations $\Delta B_k$ and $\Delta G_k$ cannot be controlled independently,
as the growth and nucleation rates of both enantiomers mutually depend on mass fractions in the liquid phase.
Following the approach of~\cite{ZB2018}, we introduce a scalar variable $v$ that characterizes the deviation of relative saturations from their steady-state values and assume that
\begin{equation}
\begin{aligned}
 \Delta G_k/{\bar G_k} & = g_k v + o(|v|),\\
\Delta B_k/{\bar B_k} & = b_k v+o(|v|),
\end{aligned}
\label{bkgk}
\end{equation}
for small values of $v$. Thus the approximation of system~\eqref{sys2x2var} takes the form
\begin{equation}
\begin{aligned}
 \frac{\partial  w_k(x,t)}{\partial t} &= - \bar G_k(1+g_k v) \frac{\partial  w_k(x,t)}{\partial x} + \psi(x) w_k(x,t)-g_k\bar n_k(x) v,\; x\in (0,\ell),\\
 { w_k}|_{x=0} &= \alpha_k v,\quad k=1,2,
\end{aligned}
\label{BC-perturbed}
\end{equation}
where terms of order $o(|v|)$ are neglected and
\begin{equation}
\alpha_k=(b_k-g_k){\bar B_k}/{\bar G_k}.
\label{alphak}
\end{equation}
We assume further that the rate of change of $v$ can be controlled, i.e.
\begin{equation}
\frac{dv}{dt}=u,
\label{v_eq}
\end{equation}
and $u$ is treated as the control.

In control system~\eqref{BC-perturbed},~\eqref{v_eq}, the functions $\bar n_k(x)$ are defined by~\eqref{equilibrium2x2} and the parameters $g_k$, $\alpha_k$ are expressed from~\eqref{bkgk},~\eqref{alphak} in terms of the Taylor coefficients of $B_k$ and $G_k$.

\section{Stabilization with Scalar Input}

Similarly to the crystallization model of Section~\ref{cryst1d}, we will use weighted $L^2$-norms to construct a control Lyapunov functional candidate:
\begin{equation}
W(t) = \frac{1}{2}\sum_{k=1}^2 \int_0^\ell \rho_{k} (x) w_k^2(x,t)\,dx + \frac{\gamma v^2(t)}{2},\quad \gamma>0,\;\rho_k(x)>0.
\label{CLF2}
\end{equation}
We compute the time derivative of $W$ along the classical solutions of the nonlinear control system~\eqref{BC-perturbed},~\eqref{v_eq} by exploiting the integration by parts and assuming that $w_k(0,t)=\alpha_k v$:
\begin{equation}
\dot W = \frac{1}{2}\sum_{k=1}^2\left(W_{0k}+v W_{1k}\right) + \gamma v u,
\label{Wdot}
\end{equation}
where
$$
\begin{aligned}
W_{0k} &= \int_0^\ell (\bar G_k \rho_k' + 2\rho_k \psi)w_k^2(x,t)\, dx - \bar G_k \rho_k(\ell) w_k^2(\ell,t),\\
W_{1k} &=  g_k \int_0^\ell (\bar G_k w_k \rho_k' - 2\rho_k \bar n_k)w_k\,dx + \bar G_k (1+g_k v)\rho_k(0)\alpha_k^2 v\\
& - \bar G_k g_k \rho_k(\ell) w_k(\ell,t).
\end{aligned}
$$

To derive a stabilizing control, we choose the density functions $\rho_k(x)>0$ as solutions to the following differential equations:
\begin{equation}
{\bar G_k} \rho_k'(x)= - {2\psi(x)}\rho_k(x) - h_k(x)\rho_{k}(x),\quad x\in [0,\ell],\;k=1,2.
\label{ODEs}
\end{equation}
Our main result concerning the stability of the closed-loop system under this above choice of densities $\rho_k(x)$ is summarized below.

\begin{theorem}\label{thm2x2}
Let $h_k\in C[0,\ell]$ be such that $h_k(x)\ge h_{k0}>0$, $k=1,2$,
\begin{equation}
\rho_k(x) = \bar \rho_k \exp\left\{-\frac{1}{\bar G_k}\int_0^x(2\psi(y)+h_k(y))dy\right\},\quad \bar \rho_k>0,\;x\in [0,\ell],
\label{rhok}
\end{equation}
and let
\begin{equation}
\begin{aligned}
u = &-\frac{\varkappa v}{2} +\frac{1}{2 \gamma}\sum_{k=1}^2 \bigl\{g_k \int_0^\ell (2\bar n_k + (h_k+2\psi)w_k)w_k\rho_k dx \\
& - \bar G_k(1+g_k v)\bar \rho_k \alpha_k^2 v +  \bar G_k g_k \rho_k(\ell)w_k(\ell,t) \bigr\},\quad \varkappa>0.
\end{aligned}
\label{feedback}
\end{equation}
Then the classical solutions of the closed-loop system~\eqref{BC-perturbed},~\eqref{v_eq},~\eqref{feedback} satisfy the following exponential decay estimate:
\begin{equation}
W(t)\le W(0)e^{-\omega t},\quad t \ge 0,
\label{asympt2}
\end{equation}
where $\omega = \min\{h_{10},h_{20},\varkappa\}>0$.
\end{theorem}

\begin{proof}
It is easy to see that the functions $\rho_k(x)$ defined by~\eqref{rhok} are general solutions of~\eqref{ODEs}.
Then we transform formula~\eqref{Wdot} by expressing the control $u$ from~\eqref{feedback} and the derivatives of $\rho_k$ from~\eqref{ODEs}.
As a result, the time derivative of $W$ along the trajectories of the closed-loop system~\eqref{BC-perturbed},~\eqref{v_eq},~\eqref{feedback} reads as follows:
$$
\dot W = -\frac{1}{2}\sum_{k=1}^2 \left(\int_0^\ell \rho_k h_k w_k^2 dx + \bar G_k \rho_k(\ell)w_k^2(\ell,t)\right)-\frac{\gamma \varkappa}{2}v^2.
$$
Then
$$
\dot W \le -\min\{h_{10},h_{20},\varkappa\} W,
$$
which proves the estimate~\eqref{asympt2}.
\end{proof}

\section{Conclusions}
The main theoretical contribution of this paper provides explicit control design schemes for the stabilization of the continuous crystallization model (Theorem~\ref{thm1}) and preferential crystallization of enantiomers (Theorem~\ref{thm2x2}).
While stability with respect to some integral measure of a population balance model was already analyzed in the paper~\cite{PK2014},
our results are based on the construction of quadratic Lyapunov functionals to achieve strong stability in the corresponding $L^2$-spaces.
The efficiency of the proposed controllers remains to be verified by numerical simulations and possible future experimental work.

\end{document}